\documentclass[reqno]{amsart}
\usepackage[left=2.5cm,right=2.5cm,top=2cm,bottom=2cm,includeheadfoot]{geometry}
\usepackage{amsmath, amsthm, amssymb, amstext}

\usepackage{hyperref,xcolor}
\hypersetup{
	pdfborder={0 0 0},
	colorlinks,}
\usepackage{enumitem}
\allowdisplaybreaks
\newtheorem{theorem}{Theorem}

\newtheorem{lemma}[theorem]{Lemma}

\newcommand*\diff{\mathrm{d}}

\newcommand{\norm}[1]{\left \Vert #1\right\Vert}

\newcommand{\Lp}[1]{L^{#1}(\Omega)}

\newcommand{\into}{\int_{\Omega}}

\newcommand{\close}{\overline{\Omega}}

\numberwithin{theorem}{section}
\numberwithin{equation}{section}

\newcommand{\R}{{\mathbb R}}

\newcommand{\Assg}[1]{\textup{(g)}}

\title [On Kirchhoff-type $p(\cdot)$-Laplacian problems with  sandwich-type and arbitrary growth]
{On Kirchhoff-type $p(\cdot)$-Laplacian problems with  sandwich-type and arbitrary growth}

	\author[K. Ho]{Ky Ho}
\address{Ky Ho\newline
	Department of Mathematics and Statistics, University of Economics Ho Chi Minh City, 59C, Nguyen Dinh Chieu St., Ho Chi Minh City, Vietnam}
\email{kyhn@ueh.edu.vn}

\subjclass[2020]{35B33, 35B45, 35J20, 35J62, 46E35}
\keywords{$p(\cdot)$-Laplacian,  Kirchhoff term, sandwich-type growth, arbitrary growth, variational methods.}

\begin{document}

\begin{abstract}
	We establish the existence of a positive bounded weak solution for a class of Kirchhoff-type $p(\cdot)$-Laplacian problems involving an arbitrary growth and a sandwich-type growth $s(\cdot)\in (\inf p,\sup p)$. This setting leads to substantial analytical difficulties in the variational analysis of the associated energy functional. By combining truncation arguments with a priori estimates, we prove the existence result under suitable assumptions on the data. 
\end{abstract}

\maketitle

\section{Introduction and Main result}

In this paper, we investigate the existence of a positive bounded weak solution to the following Kirchhoff-type $p(\cdot)$-Laplacian problem
	\begin{equation}\label{Eq}
	\begin{cases}
		-M\left(\int_\Omega\frac{1}{p(x)}|\nabla u|^{p(x)}\diff x\right)\operatorname{div}\left(|\nabla u|^{p(x)-2}\nabla u\right)=\lambda \omega(x)|u|^{s(x)-2}u+\theta b(x)|u|^{\gamma(x)-2}u& \quad \text{in } \Omega,\\
		u=0&\quad \text{on } \partial \Omega, 
	\end{cases}
\end{equation}
where $\Omega $ is a bounded domain in $\mathbb{R}^{N}$ ($N\geq 2$) with a Lipschitz boundary $\partial \Omega$, and $\lambda,\theta$ are real parameters.

In the following, for $h \in C(\close)$ we denote $h^{-}:=\inf_{x\in \close} h(x)$, $h^{+}:=\sup_{x\in \close} h(x)$, and  $h^*(x):=\frac{Nh(x)}{N-h(x)}$ if $h(x)<N$ and $h^*(x):=+\infty$ if $h(x)\geq N$ for $x\in\close$. We suppose the subsequent hypothesis: 

\vskip5pt
\begin{enumerate}[label=\textnormal{(H)},ref=\textnormal{H}]
	\item\label{H}
	$p,s,\gamma\in C(\overline{\Omega})$, $M\in C([0,\infty),\R)$, and $\omega,b\in L^\infty(\Omega)$ such that the following conditions are satisfied:
	\begin{enumerate}[label=\textnormal{(\roman*)},ref=\textnormal{\roman*}]
		\item $1<p^-<s^-\leq s^+<p^+$ and $s(x)<\min \left\{\gamma(x), p^*(x)\right\}$ for all $x\in\overline{\Omega}$;
		\item  for any $t_*>0$, there exists $T_*>t_*$ such that $0<\inf_{t\geq 0}M(t)\leq \max_{t\in [0,T_*]}M(t)=M(T_*)$;
		\item  there exists a nonempty open ball $B\subset \Omega$ such that $s(x)<p(x)$ for all $x\in B$ and $\omega(x)>0$ for a.a. $x\in B$.		
	\end{enumerate}
\end{enumerate}
Note that assumption \eqref{H}(ii) is automatically satisfied if $M \in C([0,\infty),\mathbb{R})$ is nondecreasing and  satisfies $M(0)>0$. More generally, it holds for any function $M \in C([0,\infty),\mathbb{R})$ for which there exists some $T_0>0$ such that
$
0< \min_{t\in[0,T_0]} M(t) \le \max_{t\in[0,T_0]} M(t)=M(T_0),
$ and $M$ is nondecreasing on $[T_0,\infty)$.

\par
Problems driven by the $(p,q)$-Laplacian ($1<p<q$) with reaction terms of $p$-sublinear  or $q$-superlinear growth have been extensively studied. However, relatively few results are available for the intermediate regime between $p$ and $q$, that is, a sandwich-type growth, since this setting introduces significant difficulties in the analysis of the associated energy functional. In this direction, we mention \cite{BBF.2021,HS.AML21} for critical $(p,q)$-Laplacian problems, \cite{CF,FFHW} for critical double phase problems, and \cite{BFG} for critical $(p,q)$-fractional problems. For the case of $p(\cdot)$-Laplacian, which is a borderline case of the $(p,q)$-Laplacian, we refer to \cite{MR07} for subscritical problems and \cite{HS.ANA23} for critical problems. Although \cite{HS.ANA23} addresses problems governed by an operator that generalizes the $(p,q)$-Laplacian and the $p(\cdot)$-Laplacian, its result does not cover the regime $p^-<s(\cdot)<p^+$. 

\par
In this paper, we complement the results in this direction by studying problem \eqref{Eq} under hypothesis \eqref{H}. By a bounded weak solution to problem \eqref{Eq} we mean a function $u\in W_0^{1,p(\cdot)}(\Omega)\cap L^\infty(\Omega)$ such that
\begin{multline*}
M\left(\int_\Omega\frac{1}{p(x)}|\nabla u|^{p(x)}\diff x\right)\int_\Omega |\nabla u|^{p(x)-2}\nabla u\cdot\nabla v\,\diff x\\
=\int_\Omega \left[\lambda \omega(x)|u|^{s(x)-2}u+\theta b(x)|u|^{\gamma(x)-2}u\right]v\,\diff x,\quad \forall v\in W_0^{1,p(\cdot)}(\Omega).
\end{multline*}
Our main result is stated as follows.
\begin{theorem}\label{Theo.Sandwich} Assume that \eqref{H} holds, and let $\bar{C}_s$ be defined by \eqref{Cs}. Then, for each $\lambda\in (0,\bar{C}_s),$ there exists $\theta_\star>0$ such that, for any $\theta\in (-\theta_\star,\theta_\star)$, problem~\eqref{Eq} admits a nontrivial nonnegative bounded weak solution with negative energy. Moreover, if we additionally assume that $p\in C^{1}({\close})$ and that $\omega$ and $b$ are nonnegative a.e. in $\Omega$, then for any $\theta\in [0,\theta_\star)$, problem \eqref{Eq} admits a positive bounded weak solution of class $C^{1}(\Omega)$ with negative energy. 
\end{theorem}

\par
To the best of our knowledge, all existing works in this direction have considered only nonlinearities with growth not exceeding the critical level, and Theorem \ref{Theo.Sandwich} is even new in the case $M(\cdot)\equiv 1$ and $\gamma=p^*$. Moreover, \cite[Theorem 2.1]{MR07} follows as a particular case of Theorem~\ref{Theo.Sandwich} by taking $M(\cdot)\equiv 1$, $\omega(\cdot)\equiv 1$, and $\theta=0$. We also note that, even in this setting, we obtain an explicit range of admissible values of $\lambda$, whereas in \cite{MR07} the corresponding result is established only for sufficiently small values of $\lambda$. To overcome the difficulties arising from the arbitrary growth, we combine truncation arguments with a priori estimates.

\par
The paper is organized as follows. In Section~\ref{Pre}, we briefly review the definition and basic properties of generalized Lebesgue-Sobolev spaces with variable exponents.  Section~\ref{sandwich} is devoted to the existence of  positive bounded weak solutions to problem \eqref{Eq}.


\section{Preliminaries}\label{Pre}

In this section, we briefly review the definition and basic properties of generalized Lebesgue-Sobolev spaces with variable exponents.

Let $\Omega$ be a bounded domain in $\mathbb{R}^N$ with Lipschitz boundary $\partial\Omega$, and let $M(\Omega)$ denote the set of all measurable functions $u\colon \Omega\to\R$.  Let $p\in C(\close)$ with $p^->1$, then the generalized Lebesgue space $L^{p(\cdot)}(\Omega)$ is defined by
\begin{align*}
	L^{p(\cdot)}(\Omega):=\left \{u \in M(\Omega) \,:\,\rho(u):=\into |u|^{p(x)}\,\diff x < +\infty \right \},
\end{align*}
endowed with the Luxemburg norm
\begin{align*}
	\|u\|_{p(\cdot)}:= \inf \left \{ \tau >0 : \rho\left(\frac{u}{\tau}\right) \leq 1  \right \}.
\end{align*}
The generalized Sobolev space $W^{1,p(\cdot)}(\Omega)$ is defined by
\begin{align*}
	W^{1,p(\cdot)}(\Omega)
	:=\left \{u \in L^{p(\cdot)}(\Omega) \,:\,|\nabla u| \in L^{p(\cdot)}(\Omega) \right \},
\end{align*}
endowed with the norm
\begin{align*}
	\|u\|_{1,p(\cdot)} := \|u\|_{p(\cdot)}+\|\nabla u\|_{p(\cdot)},
\end{align*}
where $\|\nabla u\|_{p(\cdot)}=\| \, |\nabla u| \,\|_{p(\cdot)}$. The space $W_0^{1,p(\cdot)}(\Omega)$ is defined as the completion of $C^\infty_0(\Omega)$ in $W^{1,p(\cdot)}(\Omega)$. 

 In view of \cite{Diening,FZ}, it holds that $L^{p(\cdot)}(\Omega)$, $W^{1,p(\cdot)}(\Omega)$ and $W_0^{1,p(\cdot)}(\Omega)$ are reflexive Banach spaces, and we can equip $W_0^{1,p(\cdot)}(\Omega)$ with the equivalent norm
\begin{align*}
	\|\cdot\|:=\|\nabla \cdot \|_{p(\cdot)}.
\end{align*} 
Moreover, it holds that
\begin{equation}\label{norm-modular}
	\min\left\{\norm{u}_{p(\cdot)}^{p^-},\norm{u}_{p(\cdot)}^{p^+}\right\}\leq \rho(u)\leq \max\left\{\norm{u}_{p(\cdot)}^{p^-},\norm{u}_{p(\cdot)}^{p^+}\right\},\quad \forall u\in \Lp{p(\cdot)},
\end{equation}
and we have the following compact embedding:
\begin{align}\label{Compact.Imb}
	W_0^{1,p(\cdot)}(\Omega) \hookrightarrow \hookrightarrow\Lp{r(\cdot)},
\end{align}
for any $r \in C(\close)$ with $ 1 \leq r(x) < p^*(x)$ for all $x\in \close$.


\section{Proof of the main result}\label{sandwich}
Throughout this section, we assume that hypothesis \eqref{H} holds. Denote $$X:=\left(W_0^{1,p(\cdot)}(\Omega),\|\cdot\|\right).$$
We also denote by $X^*$ the dual space of $X$ and by $\langle \cdot,\cdot\rangle$ the duality pairing between $X^*$ and $X$.

\vskip5pt
In view of \eqref{H} and \eqref{Compact.Imb} we have
\begin{equation}\label{BC}
	C_s:=\sup_{u\in X, \|u\|=1}\|u\|_{s(\cdot)}\in (0,\infty).
\end{equation}
Thus, we can introduce the following positive constant involving data:
\begin{equation}\label{Cs}
	\bar{C}_s:=\frac{m_0s^-}{p^+}w_0^{-1}\min \left\{C_s^{-s^+},C_s^{-s^-}\right\},
\end{equation}
where $$m_0:=\inf_{t\geq 0}M(t)\quad \text{and}\quad w_0:=\|\omega\|_\infty+\|b\|_\infty.$$

Let $\lambda,\theta\in\R$, and let $t_0>0$ be fixed and specified later such that $M(t_0)=\max_{t\in [0,t_0]}M(t)$. We define modified Kirchhoff functions $M_0, \widehat{M}_0\colon [0,\infty) \to \R$ by
\begin{equation*}
	M_0(t):=
	\begin{cases}
		M(t) &\text{if }0\leq t\leq t_0,\\
		M(t_0) &\text{if }t>t_0,
	\end{cases}
	\quad\text{and}\quad
	\widehat{M}_0(t):=\int_0^t M_0(\tau)\,\diff \tau.
\end{equation*}
It is clear that $M_0\in C([0,\infty),\R)$ and
\begin{align}\label{Est.M0}
	m_0\leq M_0(t)\leq M(t_0) \ \ \text{and}\ \ 
	m_0t\leq \widehat{M}_0(t)\leq M(t_0)t\quad \text{for all }t\in[0,\infty).
\end{align}
Next, we modify the reaction term as follows. Let $T>1$ be fixed and specified later and define  
$g:\Omega \times \R \to \R$ by
\begin{equation*}
	g(x,t):=	\begin{cases}
		0 \ & \text{ if }  \  (x,t)\in \Omega\times (-\infty,0),\\
		\lambda\omega(x)t^{s(x)-1} +  \theta b(x)t^{\gamma(x)-1} \  & \text{ if }  \  (x,t)\in \Omega\times [0, T], \\
		\lambda\omega(x)t^{s(x)-1} +  \theta b(x) T^{\gamma(x)-s(x)}  t^{s(x)-1}  \  &\text{ if }  \   (x,t)\in \Omega\times (T,\infty).
	\end{cases} 
\end{equation*}
Obviously, the modified function $g$ satisfies the following condition.
\begin{enumerate}[label=\textnormal{(G)},ref=G]
	\item \label{g} $g:\Omega\times\R\to\R$ is a  Carath\'edory function satisfying 
	$$|g(x,t)|  \leq w_0\left(|\lambda| + |\theta| T^{(\gamma-s)^+}\right)  |t|^{s(x)-1}\quad \text{for  a.a. } x \in \Omega\ \ \text{and all}\ \ t \in \R.$$
	\end{enumerate}
	Define
	$$\displaystyle G(x,t):= \int_0^t g(x,\tau) \diff \tau\quad \text{for  a.a. } x \in \Omega\ \ \text{and all}\ \ t \in \R.$$
	It is clear that 
\begin{equation}\label{G}
	G(x,t)=	\begin{cases}
		0 & \text{ if }  \  (x,t)\in \Omega\times (-\infty,0),\\
		\frac{\lambda\omega(x)}{s(x)}t^{s(x)} +  \frac{\theta b(x)}{\gamma(x)}t^{\gamma(x)}  & \text{ if }  \  (x,t)\in \Omega\times [0, T], \\
		\frac{\lambda\omega(x)}{s(x)}t^{s(x)} +  \frac{\theta b(x) T^{\gamma(x)-s(x)}}{s(x)}  t^{s(x)}-\frac{(\gamma(x)-s(x))\theta b(x)}{s(x)\gamma(x)}T^{\gamma(x)}  &\text{ if }  \   (x,t)\in \Omega\times (T,\infty).
	\end{cases} 
\end{equation}
Consequently, it holds that
\begin{equation}\label{G.bound}
	|G(x,t)| \leq \frac{w_0\left(|\lambda|+2|\theta| T^{(\gamma-s)^+}\right)}{s^-}|t|^{s(x)}\quad \text{for  a.a. } x \in \Omega\ \ \text{and all}\ \ t \in \R.
\end{equation}
We consider the following modified problem
\begin{equation}\label{Eq-m} 
	\begin{cases}
		-M_0\left(\int_\Omega\frac{1}{p(x)}|\nabla u|^{p(x)}\diff x\right)\operatorname{div}\left(|\nabla u|^{p(x)-2}\nabla u\right)=g(x,u)& \quad \text{in } \Omega,\\
		u=0&\quad \text{on } \partial \Omega. 
	\end{cases}
\end{equation}
By a weak solution to problem \eqref{Eq-m} we mean a function $u\in X$ satisfying
\begin{equation*}
	M_0\left(\int_\Omega\frac{1}{p(x)}|\nabla u|^{p(x)}\diff x\right)\int_\Omega |\nabla u|^{p(x)-2}\nabla u\cdot\nabla v\,\diff x
	=\int_\Omega g(x,u)v\,\diff x,\quad \forall v\in X.
\end{equation*}
The energy functional associated with problem \eqref{Eq-m} is defined by
\begin{equation*}\label{J}
	\begin{gathered}
		J:\, X\to\mathbb{R}, \\
		J(u):=\widehat{M}_0\left(\int_\Omega\frac{1}{p(x)}|\nabla u|^{p(x)}\diff x\right)-\int_\Omega G(x,u)\,\diff x,\quad u\in X.
	\end{gathered}
\end{equation*}
By a standard argument, using the definitions of $\widehat{M}_0$ and $G$ and the embedding \eqref{Compact.Imb}, we can easily show that $J\in C^1(X,\R)$ and its Fr\'echet derivative $J'\colon X\to X^*$ is given by
\begin{align}\label{J'}
	\left\langle J' (u),v\right\rangle= M_0\left(\int_\Omega\frac{1}{p(x)}|\nabla u|^{p(x)}\diff x\right)\int_\Omega |\nabla u|^{p(x)-2}\nabla u\cdot\nabla v\,\diff x-\int_\Omega g(x,u)v\,\diff x,\quad \forall u,v\in X.
\end{align}
Clearly, any critical point $u$ of $J$ satisfying $\|u\|_\infty<T$ and $\int_\Omega\frac{1}{p(x)}|\nabla u|^{p(x)}\diff x<t_0$ is a nonnegative bounded weak solution to the original problem \eqref{Eq}. Here, we note that from the definition of $g$ and the equality $\left\langle J' (u),u_-\right\rangle=0$, where $u_-(\cdot):=-\min\{u(\cdot),0\}$, it is easy to see that any critical point $u$ of $J$ is nonnegative a.e. in $\Omega$.
\begin{lemma}\label{le.PS}  Let $\lambda,\theta\in\mathbb{R}$, and let $\{u_n\}_{n=1}^\infty$ be a  bounded sequence in $X$ satisfying $J'(u_n)\to 0$ in $X^*$ as $n\to\infty$. Then, $\{u_n\}_{n=1}^\infty$ admits a convergent subsequence.
\end{lemma}
\begin{proof}
Since $\{u_n\}_{n=1}^\infty$ is bounded in the reflexive Banach space $X$, up to a subsequence we have
	\begin{equation*}\label{PL.wc}
		u_n(x) \to u(x)  \quad  \text{for a.a.} \ \
		x\in\Omega\ \ \text{and} \ \ 
		u_n \rightharpoonup u  \quad \text{in} \ \ X\ \ \text{as}\ \ n\to\infty.
	\end{equation*}
	Then, invoking \eqref{Compact.Imb} and the Lebesgue dominated convergence theorem we easily deduce that
	\begin{equation}\label{limgn}
		\int_\Omega g(x,u_n)(u_n-u)\,\diff x\to 0\ \ \text{as}\ \ n\to\infty.
	\end{equation}
	From \eqref{Est.M0} and \eqref{J'} we obtain
	\begin{align*}\label{PS1.S+.1}
		\notag	m_0\bigg|\int_\Omega |\nabla u_n|^{p(x)-2}\nabla u_n \cdot \nabla(u_n-u)\diff x\bigg|	\leq& M_0\left(\int_\Omega\frac{1}{p(x)}|\nabla u_n|^{p(x)}\,\diff x\right)\bigg|\int_\Omega |\nabla u_n|^{p(x)-2}\nabla u_n \cdot \nabla(u_n-u)\diff x\bigg|\\
		\notag	\leq& \big|\langle J' (u_n),u_n-u\rangle\big|+\left|\int_\Omega g(x,u_n)(u_n-u)\diff x\right|.
	\end{align*}
	Using \eqref{limgn}, the boundedness of $\{u_n\}_{n=1}^\infty$ in $X$ and $J'(u_n)\to 0$ in $X^*$ as $n\to\infty$,  the last estimate yields
	\begin{align*}
	\int_\Omega |\nabla u_n|^{p(x)-2}\nabla u_n \cdot \nabla(u_n-u)\,\diff x\to 0\ \ \text{as}\ \ n\to\infty;
	\end{align*}
	hence, $u_n\to u$ in $X$ as $n\to\infty$ due to the $(S_+)$ property of the $p(\cdot)$-Laplacian. The proof is complete.

\end{proof}
\begin{lemma} \label{le2}
	For any $\lambda >0$ and $\theta\in\R$, there exists $\phi \in X$ such that $J(t\phi)<0$ for all small $t>0.$ 
\end{lemma}
\begin{proof}
	Let $\lambda >0$ and $\theta\in\R$, and let $x_0\in B$  such that $s(x_0)<\min\{p(x_0),\gamma(x_0)\}$. Since $s, p,\gamma \in C(\overline{\Omega})$, there exist a number $\eta_0$ and an open ball $B_0\ne\emptyset$ satisfying  $0<\eta_0<\frac{\min\{p(x_0),\gamma(x_0)\}-s(x_0)}{2}$, $\overline{B_0}\subset B$, and
	\begin{equation}\label{q}
		s(x)< s(x_0)+ \eta_0, \ p(x)> p(x_0)- \eta_0,\ \gamma(x)> \gamma(x_0)- \eta_0,\quad \forall  x\in B_0.
	\end{equation}
	Let $\phi \in C_0^{\infty}(\Omega)\setminus \{0\}$ such that $\operatorname{supp}(\phi)\subset B_0$ and $0\leq \phi \leq 1$ on $B_0$. Then for $t\in (0,1),$ it follows from \eqref{Est.M0}, \eqref{G} and \eqref{q} that
	$$J(t\phi)\leq a_0t^{p(x_0)-\eta_0}-b_0\lambda t^{s(x_0)+\eta_0} +c_0|\theta|t^{\gamma(x_0)-\eta_0},$$
	where $a_0:=\frac{M(t_0)\int_{B_0}|\nabla \phi|^{p(x)}\,\diff x}{p^-}>0,\ b_0:=\frac{\int_{B_0}\omega(x)\phi^{s(x)}\,\diff x}{s^+}>0$ and $c_0:=\frac{\int_{B_0}|b(x)|\,\diff x}{\gamma^-}\geq 0$. Thus, $J(t\phi)<0$ for sufficiently small $t>0$ since $s(x_0)+\eta_0<\min\{p(x_0)-\eta_0,\gamma(x_0)-\eta_0\}.$

\end{proof}

\begin{lemma}\label{le.geo}
	For each $\lambda\in (0,\bar{C}_s)$, there exists $\theta_\star>0$ such that, for any $\theta\in (-\theta_\star,\theta_\star)$, there exist $\delta,\rho>0$ satisfying
	$$\underset{u\in \partial B_{\delta}}{\inf} J(u)\geq \rho>0> \underset{u\in B_{\delta}}{\inf} J(u),$$
	where $B_{\delta}:=\{u\in X:\norm{u}\leq \delta\}$ and $\partial B_{\delta}:=\{u\in X:\norm{u}= \delta\}$.
\end{lemma}
\begin{proof} 	
	Let $\lambda\in (0,\bar{C}_s)$ and put 
	\begin{equation}\label{del}
		\delta:=\left(\frac{1+\lambda \bar{C}_s^{-1}}{2}\right)^{\frac{1}{p^+-s^-}},
	\end{equation}
where $\bar{C}_s$ is given in \eqref{Cs}. It is clear that $\delta\in (0,1)$. Thus, for $u\in \partial B_{\delta}$ by \eqref{Est.M0}, \eqref{G.bound} and \eqref{norm-modular} we obtain
	\begin{align*}
		J(u)&\geq \frac{m_0}{p^+}\|u\|^{p^+}- \frac{w_0\left(\lambda+2|\theta| T^{(\gamma-s)^+}\right)}{s^-}\max\left\{\|u\|_{s(\cdot)}^{s^+},|u\|_{s(\cdot)}^{s^-}\right\}.
	\end{align*}
	Then, by \eqref{BC}, \eqref{Cs}, \eqref{del} and noticing $\|u\|=\delta<1$ it yields
	\begin{align*}
		J(u)\geq& \frac{m_0}{p^+}\|u\|^{p^+}- \frac{w_0\left(\lambda+2|\theta| T^{(\gamma-s)^+}\right)}{s^-}\max \left\{C_s^{s^+},C_s^{s^-}\right\}\|u\|^{s^-}\\
		&=\frac{m_0\delta^{s^-}}{p^+}\left(\frac{1-\lambda \bar{C}_s^{-1}}{2}-2\bar{C}_s^{-1}T^{(\gamma-s)^+}|\theta| \right).
	\end{align*}
	Thus, by choosing 
	\begin{equation}\label{theta}
		\theta_\star=\frac{\left(\bar{C}_s-\lambda\right) T^{-(\gamma-s)^+}}{4},
	\end{equation}
	for $\theta\in (-\theta_\star,\theta_\star)$ one has
	
	\begin{equation*}\label{PB}
		J(u)\geq\rho,\ \ \forall u\in\partial B_{\delta},
	\end{equation*}
	where $\rho:=\frac{2m_0\delta^{s^-}\bar{C}_s^{-1}T^{(\gamma-s)^+}}{p^+}\left(\theta_\star-|\theta| \right)>0$. From this and Lemma~\ref{le2} the conclusion follows. 
\end{proof}

Now, fix any $h\in C(\overline{\Omega})$ with $\max\left\{p(\cdot),s(\cdot)\right\}<h(\cdot)<p^*(\cdot)$, and consider the following condition:

\vskip5pt
\begin{enumerate}[label=\textnormal{(G1)},ref=G1]
	\item \label{G1} $|g(x,t)|  \leq  w_0\bar{C}_s\left(|t|^{h(x)-1}+1\right)$ for  a.a. $x \in \Omega$ and all $t \in \R$.
\end{enumerate}

\vskip5pt
Then, we have the following lemma.
\begin{lemma}\label{le.b}
	Let \eqref{G1} hold, and let $u$ be any weak solution to problem \eqref{Eq-m}. Then, $u \in L^\infty(\Omega)$ and satisfies the following a priori bound
	\begin{equation*}
		\|u\|_{\infty} \leq C_0  \max \left\{ \|u\|^{\tau_1}_{h(\cdot)}, \|u\|^{\tau_2}_{h(\cdot)}\right\},
	\end{equation*} 
	where $C_0, \tau_1,\tau_2$ are positive constants independent of $u$ and $\theta$. 
\end{lemma}
This lemma can be proved in the same manner as in \cite[Proof of Theorem 4.2]{HW2022}. Indeed, in \cite[Proof of Theorem 4.2]{HW2022}, one puts $\mu=0$ and replaces $p^*$ by a function $m\in C(\close)$ such that $h(\cdot)<m(\cdot)<p^*(\cdot)$, therefore, the condition $p(\cdot)<N$ is not needed since we apply \eqref{Compact.Imb} instead of \cite[Proposition 3.7]{HW2022}. For this reason, we omit the details.

\vskip5pt
We are now in a position to prove Theorem~\ref{Theo.Sandwich}.

\begin{proof}[\textbf{Proof of Theorem~\ref{Theo.Sandwich}}] By \eqref{Compact.Imb}, we find $C_h>0$ such that
	\begin{equation}\label{Ch}
		\|u\|_{h(\cdot)}\leq C_h\|u\|,\quad \forall u\in X.
	\end{equation}
	Now, let $\lambda\in (0,\bar{C}_s)$, and let $\delta$ be defined by \eqref{del}. Put
	$$T=C_0\max \left\{ (C_h \delta)^{\tau_1}, (C_h\delta)^{\tau_2}\right\}+1,$$ where $C_0, \tau_1,\tau_2$ are given in Lemma~\ref{le.b}, and let $\theta_\star$ be as in \eqref{theta}.

	Let $\theta\in (-\theta_\star,\theta_\star)$ and denote  $$c:=\underset{u\in B_{\delta}}\inf J(u),$$ where $B_{\delta} :=\{u\in X: \|u\|\leq\delta\}$ with the boundary $\partial B_{\delta}$. By Lemma~\ref{le.geo}, it holds that
	$c<0<\underset{u\in\partial B_{\delta}}\inf J(u).$ Applying the Ekeland variational principle, we find a sequence $\{u_n\}_{n=1}^\infty\subset B_{\delta}$ such that
	\begin{equation}\label{PS1}
		\begin{cases}
			J(u_n)\to c \ \ \text{as} \ \ n\to \infty, \\
			J'(u_n)\to 0 \ \ \text{in} \  X^\ast \ \text{as}\ n\to \infty.
		\end{cases}
	\end{equation}
Thus, by Lemma~\ref{le.PS}, $\{u_n\}_{n=1}^\infty$ admits a subsequence $\{u_{n_k}\}_{k=1}^\infty$ such that $u_{n_k}\to \underline{u}$ in $X$ as $k\to \infty.$  It follows from  this and \eqref{PS1} that $J(\underline{u})= c$ and $J'(\underline{u})= 0$. Hence, $\underline{u}$ is a nontrivial nonnegative weak solution to problem \eqref{Eq-m} with  $J(\underline{u})=c<0$ and $\|\underline{u}\|\leq \delta$. On the other hand,  \eqref{G1} is fulfilled thanks to condition \eqref{g} and estimate $\lambda + |\theta| T^{(\gamma-s)^+}<\bar{C}_s$. Thus, by Lemma \ref{le.b} and \eqref{Ch}, we obtain
	\begin{equation}\label{E-T}
		\|\underline{u}\|_{\infty} \leq C_0  \max \left\{ \|\underline{u}\|^{\tau_1}_{h(\cdot)}, \|\underline{u}\|^{\tau_2}_{h(\cdot)}\right\}\leq C_0\max \left\{ (C_h \delta)^{\tau_1}, (C_h \delta)^{\tau_2}\right\}<T.
	\end{equation} 
	On the other hand, using $\left\langle J' (\underline{u}),\underline{u}\right\rangle=0$ we estimate
	\begin{align*}
		m_0p^-\int_\Omega \frac{1}{p(x)}|\nabla \underline{u}|^{p(x)}\,\diff x\leq  M_0\left(\int_\Omega \frac{1}{p(x)}|\nabla \underline{u}|^{p(x)}\,\diff x\right)\int_\Omega |\nabla \underline{u}|^{p(x)}\,\diff x=\int_\Omega g(x,\underline{u})\underline{u}\,\diff x.
	\end{align*}
	Combining this with \eqref{g} gives 
	\begin{align}\label{E-t01}
		\int_\Omega \frac{1}{p(x)}|\nabla \underline{u}|^{p(x)}\,\diff x\leq\frac{w_0\left(\lambda + |\theta| T^{(\gamma-s)^+}\right)}{m_0p^-}\int_\Omega |\underline{u}|^{s(x)}\,\diff x\leq\frac{|\Omega|w_0 \left(\lambda + |\theta| T^{(\gamma-s)^+}\right)T^{s^+}}{m_0p^-},
	\end{align}
	where $|\Omega|$ denotes the Lebesgue measure of $\Omega$. Noticing $0<T\leq C_0\max\left\{C_h^{\tau_1},C_h^{\tau_2}\right\}+1$ and $0<\lambda + |\theta| T^{(\gamma-s)^+}<\overline{C}_s$ we have 
	\begin{equation}\label{E-t02}
		\frac{|\Omega|w_0 \left(\lambda + |\theta| T^{(\gamma-s)^+}\right)T^{s^+}}{m_0p^-}\leq \frac{|\Omega|w_0 \overline{C}_s\left(C_0\max\left\{C_h^{\tau_1},C_h^{\tau_2}\right\}+1\right)^{s^+}}{m_0p^-}.
	\end{equation}
	We now fix $t_0>\frac{|\Omega|w_0 \overline{C}_s\left(C_0\max\left\{C_h^{\tau_1},C_h^{\tau_2}\right\}+1\right)^{s^+}}{m_0p^-}$ such that $M(t_0)=\max_{t\in [0,t_0]}M(t)$.
		Thus, from \eqref{E-T} -\eqref{E-t02}, $\underline{u}$ is a nontrivial nonnegative bounded  weak solution to problem \eqref{Eq} with negative energy. 
		
		\vskip5pt
		If we assume additionally that $p\in C^{1}({\close})$ and $\omega$ and $b$ are nonnegative a.e. in $\Omega$, then $\underline{u}\in C^1(\Omega)$ in view of \cite{Fan2007}. Utilizing this fact, we apply the strong maximum principle for the $p(\cdot)$-Laplacian (see \cite{Fan-SMP}) to deduce the positivity of $\underline{u}$ in $\Omega$. The proof is complete.

\end{proof}

\subsection*{Acknowledgment}
This research was supported by the University of Economics Ho Chi Minh City (UEH), Vietnam.



\begin{thebibliography}{99}
	
	\bibitem{BBF.2021}L. Baldelli, Y. Brizi, R. Filippucci, Multiplicity results for $(p, q)$-Laplacian equations with critical exponent in $\R^N$ and negative energy, 
	Calc. Var. Partial Differ. Equ. 60 (2021) 8, 30pp.
	
	\bibitem{BFG} M. Bhakta, A.Fiscella, S. Gupta, Critical $(p,q)$-fractional problems involving a sandwich type nonlinearity, Bull. Sci. Math. 208 (2026) 103786.
	
		
	\bibitem{CF} F. Colasuonno, K. Perera, Critical growth double phase problems: The local case and a Kirchhoff type case, J. Differ. Equ. 422 (2025) 426--488.
	
		
\bibitem{Diening} L. Diening,  P. Harjulehto, P. H\"{a}st\"{o}, M. R\^{u}\v{z}i\v{c}ka, 	Lebesgue and Sobolev spaces with  variable exponents, 	Lecture Notes in Mathematics  2017, Springer-Verlag, Heidelberg, 2011.

	\bibitem{Fan2007} X. Fan, Global $C^{1,\alpha}$ regularity for variable exponent elliptic equations in divergence form, J. Differ. Equ. 235 (2007) 397--417.
	
	\bibitem{FZ} X. Fan, D. Zhao,
	On the spaces $L^{p(x)}(\Omega)$ and $W^{m,p(x)}(\Omega)$, J. Math. 	Anal. Appl. 263 (2001) 424--446.
	
	\bibitem{Fan-SMP} X. Fan, Y. Zhao, Q. Zhang,
	A strong maximum principle for $p(x)$-Laplace equations,
	Chinese J. Contemp. Math. 24 (3) (2003) 277--282.
	
	\bibitem{FFHW} C. Farkas, A. Fiscella, K. Ho, P. Winkert, Critical double phase problems involving sandwich-type nonlinearities, Discrete Contin. Dyn. Syst. 48 (2026) 352--375.
	
	
	\bibitem{HS.AML21} K. Ho, I. Sim, An existence result for $(p, q)$-Laplace equations involving sandwich-type and critical growth, Appl. Math. Lett. 111 (2021) 106646.
	
	\bibitem{HS.ANA23} K. Ho, I. Sim, On sufficient ``local'' conditions for existence results to generalized $p(\cdot)$-Laplace equations involving
	critical growth, Adv. Nonlinear Anal. 12 (1) (2023) 182--209.
	
	\bibitem{HW2022} K. Ho,  P. Winkert, New embedding results for double phase problems with variable exponents and a priori bounds for corresponding generalized double phase problems, Calc. Var. Partial Differ. Equ. 62 (2023) 227, 38 pp.
		
	\bibitem{MR07} M. Mih\u{a}ilescu, V. R\u{a}dulescu, On a nonhomogeneous quasilinear eigenvalue problem in Sobolev spaces with variable exponent, Proc. Amer. Math. Soc. 135 (9) (2007) 2929--2937.
	

\end{thebibliography}
\end{document}